\documentclass[12pt]{article}
\usepackage{latexsym,amsmath,amssymb}

\newif\ifdviwin

\usepackage{enumerate}
\usepackage[latin1]{inputenc}
\usepackage[english]{babel}
\usepackage{indentfirst}
\usepackage[mathscr]{eucal}
\usepackage{amssymb,amsmath,amsfonts}
\usepackage{graphicx}

\newif\ifdviwin

\dviwintrue

\def\cA{\mathcal{A}}

\def\cK{\mathcal{K}}

\def\N{\mathbb{N}}
\def\L{\mathbb{L}}

\def\R{\mathbb{R}}

\def\C{\mathbb{C}}
\def\D{\mathbb{D}}

\def\S{\mathbb{S}}
\def\X{\widetilde{X}}

\def\ve{\varepsilon}
\def\Sg{\Sigma}

\newcommand{\longui}{\operatorname{length}}
\newcommand{\dist}{\operatorname{dist}}

 \newtheorem{theorem}{Theorem}
 \newtheorem{proposition}{Proposition}
 \newtheorem{corollary}{Corollary}
 \newtheorem{lemma}{Lemma}
 \newtheorem{claim}{Claim}

 \newenvironment{proof}{\rm \trivlist \item[\hskip \labelsep{\it
      Proof.}]}{\nopagebreak \hfill $\Box$ \endtrivlist}

\begin{document}
\mbox{}\vspace{0.4cm}\mbox{}

\begin{center}
\rule{14cm}{1.5pt}\vspace{0.5cm}

{\Large \bf Complete minimal surfaces in $\R^3$} \\ [0.3cm]{\Large \bf with a prescribed coordinate function }\\
\vspace{0.5cm} {\large Antonio Alarc\'{o}n\footnote{Research
partially supported by MEC-FEDER grant number MTM2007-61775.} and
Isabel Fern\'{a}ndez\footnote{Research partially supported by
MEC-FEDER grant number MTM2007-64504.} }

\vspace{0.3cm} \rule{14cm}{1.5pt}
\end{center}

  \vspace{0.5cm}

\begin{abstract}
In this paper we construct complete simply connected minimal
surfaces with a prescribed coordinate function. Moreover, we prove
that these surfaces are dense in the space of all minimal surfaces
with this coordinate function (with the topology of the smooth
convergence on compact sets).
\\


\end{abstract}


\section{Introduction}\label{sec:intro}

An isometric immersion $X:M\to\R^3$ of a Riemann surface into
space is said to be minimal if its coordinate functions are
harmonic on $M$. In 1980, Jorge and Xavier \cite{JX} constructed a
complete minimal surface contained in a slab of $\R^3$, disproving
a conjecture by Calabi \cite{C}. They looked for a minimal
immersion of the disk with third coordinate $x_3(z)=\mbox{Re}(z)$
and complete metric.\\

In light of the above, it appears as a natural question whether
any harmonic function can be realized as a coordinate of a
complete minimal surface. The present paper is devoted to answer
this question in the simply connected case. More specifically, we
extend Jorge-Xavier's result to prove that any harmonic function
defined on a simply connected domain is a coordinate function of a
conformal complete minimal immersion (see Theorem
\ref{th:corolario}). Moreover, we show that complete surfaces are
dense in the space of (simply connected) minimal surfaces with a
prescribed coordinate (Corollary \ref{co:dens}). These results
come as a consequence of the following one.

\begin{theorem}\label{th:main}
Let $X=(X_1,X_2,X_3):\Sg\to\R^3$ be a conformal minimal immersion
on $\Sg=\C,\D$, with $X_3$ being non-constant. Consider $K\subset
\Sg$ a compact set and $\ve>0$.

Then, there exists a complete conformal minimal immersion $Y=(Y_1,Y_2,Y_3):\Sg\to\R^3$ such that
\begin{enumerate}[\rm (a)]
\item $\|X-Y\|<\ve$ in $K$.

\item $X_3=Y_3$.
\end{enumerate}
\end{theorem}

We also derive from Theorem \ref{th:main} some results concerning
existence of complete null curves in $\C^3$ and complete maximal
surfaces in $\L^3$ with a prescribed coordinate (Section \ref{sec:cons}).\\

The construction of the Jorge-Xavier's surface relies on a clever
use of the Runge's classical theorem with a suitable labyrinth
close to the boundary of the disk. A refinement of Jorge and
Xavier's ideas led to Nadirashvili \cite{N} to construct conformal
complete bounded minimal disks. Nadirashvili's arguments have
given rise to the construction of complete bounded minimal
surfaces with other additional properties (see for instance
\cite{LMM,MN,AFM,A}). However, all the coordinate functions of
these examples are implicit.

Despite we use some ideas related to Nadirashvili's technique in
the proof of the above theorem, it is not possible in general to
construct complete bounded minimal surfaces with a prescribed
(bounded) coordinate function. We show a requirement for a
harmonic function on the disk to be the coordinate function of a
complete bounded minimal surface in Proposition \ref{prop} (see
also \cite{N2,AN}).\\

Finally, we would like to point out that the only complete simply
connected embedded minimal surfaces are the plane and the
helicoid \cite{MR,CM}. Therefore, our surfaces are not embedded,
except for the aforementioned cases.


\section{Preliminaries}\label{sec:prelim}

This section is devoted to briefly summarize the notation and
results that we use in the paper.

From now on, we denote by $\Sg$ an open simply-connected Riemann surface. By the Uniformization Theorem we can assume that $\Sg$ is eihter the complex plane $\C$ or the open unit disk $\D$. Furthermore, for any $r>0$, we denote by $\D_r=\{z\in\C\;|\; |z|<r\}$ and $\S^1_r=\{z\in\C\;|\; |z|=r\}$.

Consider a Riemannian metric $d\tau^2$ in $\Sg$. Given a curve $\alpha$ in $\Sg$, by $\longui_{d\tau^2}(\alpha)$ we mean the length of $\alpha$ with respect to the metric $d\tau^2$. Moreover, we define:
\begin{itemize}
\item $\dist_{d\tau^2}(p,q)=\inf \{\longui_{d\tau^2}(\alpha) \: | \: \alpha:[0,1]\rightarrow \Sg, \; \alpha(0)=p,\alpha(1)=q \}$, for any $p,q\in \Sg$.
\item $\dist_{d\tau^2}(T_1,T_2)=\inf \{\dist_{d\tau^2}(p,q) \;|\;p \in T_1, \;q \in T_2 \}$, for any $T_1, T_2 \subset  \Sg$.
\end{itemize}
Throughout the paper, we work with metrics induced by conformal minimal
immersions $X:\Sg\to\R^3.$ Then, by $\lambda_{X}^2|dz|^2$ we
mean the Riemannian metric induced by $X$ in $\Sg.$ We
also write $\dist_{X}(T_1,T_2)$ instead of
$\dist_{\lambda_X^2|dz|^2}(T_1,T_2),$ for any sets $T_1$ and $T_2$ in $\Sg.$

\subsection{Minimal surfaces background}

The theory of complete minimal surfaces is closely related to the theory of Riemann surfaces. This is due to the fact that any such surface is given by a triple $\Phi=(\Phi_1, \Phi_2, \Phi_3)$ of holomorphic 1-forms defined on some Riemann surface $M$ such that
\begin{equation} \label{eq:conforme}
\Phi_1^2+\Phi_2^2+\Phi_3^2=0,
\end{equation}
\begin{equation}\label{eq:inmersion}
\|\Phi_1\|^2+\|\Phi_2\|^2+\|\Phi_3\|^2 \neq 0,
\end{equation}
and all periods of the $\Phi_j$ are purely imaginary.
Then the minimal immersion $X:M
\rightarrow \R^3$ can be parameterized by $z \mapsto \mbox{Re}
\int^z \Phi.$ The above triple is called the Weierstrass
representation of the immersion $X$. Usually, the first
requirement \eqref{eq:conforme} (which ensures the conformality
of $X$) is guaranteed by introducing the formulas
\[
\Phi_1 =\frac12 \left( \frac1{g}-g\right) \, \Phi_3, \quad \Phi_2 =\frac{\rm i}2 \left( \frac1{g}+g\right) \, \Phi_3,
\]
with a meromorphic function $g$ (the stereographic projection of
the Gauss map) and a holomorphic 1-form $\Phi_3$. The pair
$(g,\Phi_3)$ is called the Weierstrass data of the minimal
immersion $X$. In this article all the minimal immersions are
defined on the simply connected Riemann surface $\Sg$. Then, the
Weierstrass data have no periods and so the only requirements are
\eqref{eq:conforme} and \eqref{eq:inmersion}. The metric of $X$
can be expressed as
\begin{equation}\label{eq:metric}
\lambda^2_X|dz|^2=\frac12
\|\Phi\|^2=\left(\frac12\left(\frac1{|g|}+|g|\right) \|\Phi_3
\|\right)^2.
\end{equation}

\subsubsection{The L\'{o}pez-Ros transformation}
The proof of Lemma \ref{lem:lemma} exploits what has come to be called the L\'{o}pez-Ros transformation. If $M$ is a Riemann surface and $(g,\Phi_3)$ are the Weierstrass data of a minimal immersion $X:M \rightarrow \R^3$, we define on $M$ the data
\[
\widetilde g= \frac{g}{h}, \qquad \widetilde \Phi_3= \Phi_3,
\]
where $h:M \rightarrow \C$ is a holomorphic function without
zeros. If the periods of this new Weierstrass representation are purely imaginary, then it defines a minimal immersion $\widetilde X: M \rightarrow \R^3$. This method provides us with a powerful and natural tool for deforming minimal surfaces. From our point of view, the most important property of the resulting surface is that the third coordinate function is preserved. Note that the intrinsic metric is given by \eqref{eq:metric} as
\[
\lambda^2_{\widetilde X}|dz|^2=\left(\frac12\left(\frac{|h|}{|g|}+\frac{|g|}{|h|}\right)\, \|\Phi_3 \|
\right)^2.
\]
This means that we can increase the intrinsic distance in a prescribed compact of $M$, by using  suitable functions $h$.


\section{Proof of Theorem \ref{th:main}}\label{sec:main}

In order to prove the main theorem we need the following technical lemma. It will be proved later in subsection \ref{sub:lemma}.

\begin{lemma}\label{lem:lemma}
Let $X=(X_1,X_2,X_3):\Sg\to\R^3$ be a conformal minimal immersion
being $X_3$ non-constant. Consider two positive constants $0<r<R$
(with $R<1$ if $\Sg=\D$).

Then, for any $\ve,s>0$ there exists a conformal minimal
immersion $\X=(\X_1,\X_2,\X_3):\Sg\to\R^3$ such that
\begin{enumerate}[\rm ({L}1)]
\item $\dist_{\X}(0,\S^1_R)>s$.

\item $\|\X-X\|<\ve$ in $\overline\D_r$.

\item $\X_3=X_3$.
\end{enumerate}
\end{lemma}

Assuming the above lemma, the proof of Theorem \ref{th:main} goes as follows. First of all, consider $r_0>0$ ($0<r_0<1$ in case $\Sg=\D$) such that $K\subset \D_{r_0}\subset \Sg$. Let $\{r_n\}_{n\in\N}$ be an increasing sequence of positives, with $r_1=r_0$, and such that $\{r_n\}\nearrow +\infty$ in case $\Sg=\C$, and $\{r_n\}\nearrow 1$ in case $\Sg=\D$. Finally, take any sequence $\{\sigma_n\}_{n\in\N}$ with $0<\sigma_n<1$ and so that $\prod_{k=1}^\infty\sigma_k = 1/2$.

We will obtain the immersion $Y$ as a limit of a sequence of immersions $\{X_n\}_{n\in\N}$. For any $n\in\N$, we will construct a family $\chi_n=\{X_n,\ve_n\}$ where $X_n:\Sg\to\R^3$ is a conformal minimal immersion and $\ve_n<6\ve/(n^2\pi^2)$ is a positive number. Furthermore, the sequence $\{\chi_n\}_{n\geq 2}$ will satisfy the following properties.
\begin{enumerate}[\rm (A$_{n}$)]
\item $\|X_n-X_{n-1}\|<\ve_n$ in $\overline{\D}_{r_{n-1}}$.

\item $\dist_{X_n}(0,\S^1_{r_n})>n$.

\item $\lambda_{X_n}\geq \sigma_n\cdot \lambda_{X_{n-1}}$ in $\overline{\D}_{r_{n-1}}$.

\item $(X_n)_3=X_3$.
\end{enumerate}

The sequence is constructed in a recursive way. The first element of the sequence is the immersion $X_1=X$ and any positive $\ve_1<6\ve/\pi^2$. Assume we have defined $\chi_1,\ldots,\chi_n$. Let us show how to construct the family $\chi_{n+1}$. Consider a sequence $\{\xi_m\}_{m\in\N}$ decreasing to zero and such that
\[
\xi_m<\min \left\{\ve_n\,,\,\frac{6\ve}{\pi^2(n+1)^2}\right\},\quad \forall m\in\N.
\]
Let $F_m:\Sg\to\R^3$ be the immersion obtained from Lemma \ref{lem:lemma} for the data
\[
X=X_n,\quad r=r_n,\quad R=r_{n+1},\quad \ve=\xi_m,\quad s=n+1.
\]
The sequence $\{F_m\}_{m\in\N}$ converges to $X_n$ uniformly on $\overline{\D}_{r_n}$. Therefore, there exists $m_0\in\N$ large enough so that
\begin{equation}\label{equ:metricas}
\lambda_{F_{m_0}}\geq \sigma_{n+1}\cdot \lambda_{X_n}\quad \text{in }\overline{\D}_{r_{n-1}}.
\end{equation}
Recall that $0<\sigma_{n+1}<1$. Label $X_{n+1}:=F_{m_0}$ and $\ve_{n+1}:=\xi_{m_0}$. Properties (L2), (L1) and (L3) imply (A$_{n+1}$), (B$_{n+1}$) and (D$_{n+1}$), respectively. Finally, inequality \eqref{equ:metricas} gives (C$_{n+1}$). This concludes the construction of the sequence $\{\chi_n\}_{n\in\N}$.

Since $\Sg=\cup_{n\in\N}\overline{\D}_{r_n}$, we infer from
properties (A$_n$), $n\in\N$, that $\{X_n\}_{n\in\N}$ converges
to a smooth map $Y:\Sg\to\R^3$ uniformly on compact sets of
$\Sg$. Let us check that $Y$ satisfies the conclusion of the
theorem.

\noindent{$\bullet$} $Y$ is an immersion. Indeed, consider $p\in\Sg$. Fix $n_0\in\N$ so that $p\in \overline{\D}_{r_{n_0}}$. From properties (C$_n$), $n> n_0$, we obtain that
\[
\lambda_{X_n}(p)\geq \frac12 \cdot\lambda_{X_{n_0}}(p),
\]
where we have used that $\prod_{k=1}^\infty \sigma_k=1/2$. If we take limits in the above inequality as $n\to\infty$ we infer that $\lambda_Y(p)\geq \frac12\lambda_{X_{n_0}}(p)>0$, and therefore, $Y$ is an immersion.

\noindent{$\bullet$} Since $Y$ is harmonic (Harnack's Theorem), it is minimal and conformal.

\noindent$\bullet$ $Y$ is complete. In order to check it, let $\alpha$ be a divergent curve in $\Sg$ starting at $0$. Then, for any $k\in\N$, we have
\[
\longui_{Y}(\alpha)\geq \longui_{Y}(\alpha\cap
\overline{\D}_{r_k})\geq \frac12 \longui_{X_k}(\alpha\cap
\overline{\D}_{r_k})>\frac{k}2,
\]
where we have used properties (B$_n$) and (C$_n$), $n\geq k$.
Hence, $\longui_{Y}(\alpha)=\infty$, which proves the
completeness of $Y$.

\noindent$\bullet$ Statement (a) follows from properties (A$_n$), $n\in\N$, and the facts that $\sum_{n\geq 1}\ve_n<\ve$ and $K\subset \overline{\D}_{r_n}$, $\forall n\in\N$.

\noindent$\bullet$ Statement (b) is a trivial consequence of properties (D$_n$), $n\in\N$.

The proof is done.


\subsection{Proof of Lemma \ref{lem:lemma}}\label{sub:lemma}

Let $(g,\Phi_3)$ be the Weierstrass data of the immersion $X$.
Since $X_3$ is non-constant and $\Sg$ is simply connected we can
write $\Phi_3=\phi_3(z)dz$ with $\phi_3$ non identically zero.
Therefore, there exist a constant $\delta>0$ and two real numbers
$r'$ and $R'$ with $r<r'<R'<R$, satisfying
\begin{equation}\label{eq:nocero}
|\phi_3|>\delta \quad \text{in }\D_{R'}\setminus
\overline{\D}_{r'}.
\end{equation}

Fix a natural $N$ (which will be specified later) such that
$2/N<R'-r'$. The immersion $\X$ will be obtained from $X$ by
using L\'{o}pez-Ros transformation. The effect of this deformation
will be concentrated on a labyrinth of compact sets contained in
$\D_{R'}\setminus \overline{\D}_{r'}$. On the other hand, the
deformation hardly acts on $\overline\D_r$. The shape of the
labyrinth is inspired in those used by Jorge and Xavier
\cite{JX}. Let us describe it.

For any $n\in\N$, $n=1,\ldots,2N^2$, define $s_n=R'-n/N^3$ and
label $s_0=R'$. Now, consider the set (see Figure
\ref{fig:laberinto})
\[
\cK_n=\left\{ z\in\C\; \left|\; s_n+\frac1{4N^3}\leq |z|\leq
s_{n-1}-\frac1{4N^3},\quad \frac1{N^2}\leq {\rm arg}((-1)^{n}z)
\leq 2\pi-\frac1{N^2}\right.\right\}.
\]
\begin{figure}[ht]
    \begin{center}
    \scalebox{0.5}{\includegraphics{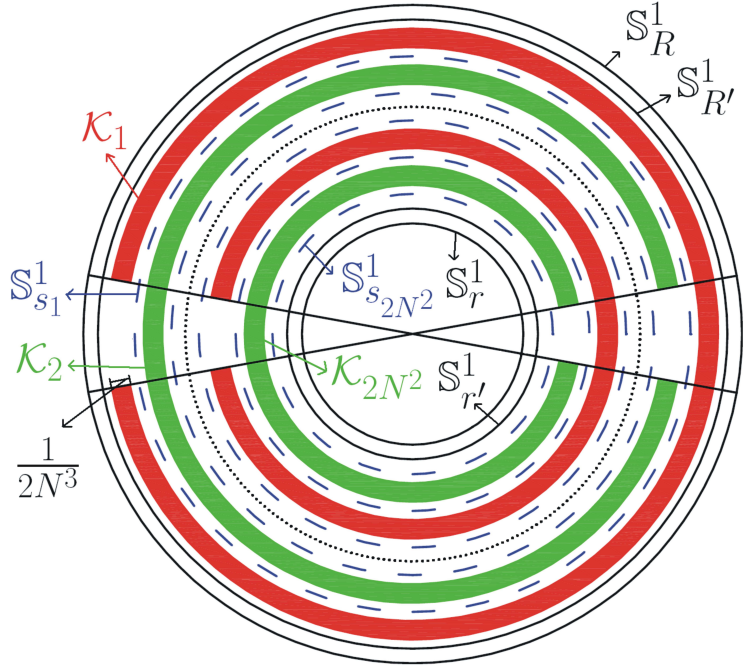}}
        \end{center}
\caption{The labyrinth of compact sets.}\label{fig:laberinto}
\end{figure}

Then, consider
\[
\cK=\bigcup_{n=1}^{2N^2}\cK_n.
\]

From the definition of the compact set $\cK$ follows that any
curve joining $\S^1_{r'}$ and $\S^1_{R'}$ without going through
$\cK$ must have large Euclidean length. This fact is stated in the
following claim.

\begin{claim}\label{claim:lemma}
Let $\lambda^2|dz|^2$ be a conformal metric on $\Sg$ satisfying
\[
\lambda\geq
\begin{cases}
c & \text{in }\D_{R'}\setminus\overline{\D}_{r'}
\\
c\, N^4 & \text{in }\cK,
\end{cases}
\]
for some $c>0$.

Then, there exists a positive constant $\rho$ not depending on $c$ nor $N$ and such that
\[
\longui_{\lambda^2|dz|^2}(\alpha)>\rho\cdot c\cdot N
\]
holds, for any $\alpha$ curve in $\Sg$ joining $\S^1_{r'}$ and $\S^1_{R'}$.
\end{claim}

Now, we define the function we use as parameter of the L\'{o}pez-Ros transformation.
In order to do it, for any $\beta>0$ we consider a holomorphic function $h_\beta:\Sg\to \C^*$ with the following two properties:
\begin{itemize}
\item $|h_\beta-1|<1/\beta$ in $\overline{\D}_r$.

\item $|h_\beta-\beta|<1/\beta$ in $\cK$.
\end{itemize}
The set $\cK\cup \overline{\D}_r$ is a compact set and
$\Sg\setminus (\cK\cup \overline{\D}_r)$ is connected. Hence, the
existence of the above functions is guaranteed by Runge's Theorem.

Define the Weierstrass data on $\Sg$
\[
g^\beta:=\frac{g}{h_\beta},\qquad \Phi_3^\beta :=\Phi_3.
\]
These Weierstrass data give rise to a minimal immersion
$X^\beta:\Sg\to\R^3$. Notice that $h_\beta$ converges to $1$
(resp. $\infty$) uniformly on $\overline{\D}_r$ (resp. $\cK$).
Hence, there exists a large enough $\beta_0$ such that
\begin{enumerate}[\rm (i)]
\item $\|X^{\beta_0}-X\|<\ve$ in $\overline{\D}_r$.

\item $\lambda_{X^{\beta_0}}\geq \delta\cdot N^4$ in $\cK$, where $\delta$ was defined in equation
\eqref{eq:nocero}, and $\lambda_{X^{\beta_0}}^2$ is the conformal
factor of the metric induced by $X^{\beta_0}$ (see
\eqref{eq:metric}).
\end{enumerate}
Label $\X:=X^{\beta_0}$. Let us check that $\X$ has the desired
properties, provided that $N$ was chosen to be large enough.
 From the very definition of $\X$, statements (L2) and (L3) trivially hold. In order to check (L1) notice first that
\begin{equation}\label{equ:iii}
\lambda_{\X}\geq |\phi_3|>\delta\quad \text{in }
\D_{R'}\setminus\overline{\D}_{r'},
\end{equation}
where we have taken into account \eqref{eq:nocero}. Then,
properties (ii), \eqref{equ:iii} and Claim \ref{claim:lemma}
guarantee that
\[
\longui_{\X}(\alpha)\geq \delta\cdot\rho\cdot N,
\]
for any curve $\alpha$ in $\Sg$ joining $0$ with $\S^1_R$. Assume that we chose $N$ large enough so that $\delta\cdot\rho\cdot N>s$ (recall that neither $\rho$ nor $\delta$ depend on $N$). This finishes the proof.


\section{Some consequences of Theorem \ref{th:main}}\label{sec:cons}

It may be followed from our main theorem some interesting results concerning not only minimal surfaces.

As we deal with simply connected surfaces, it is not hard to realize any harmonic function as a coordinate function of a minimal immersion. Hence, Theorem \ref{th:main} implies that any harmonic function on a simply connected domain is a coordinate function of a complete minimal immersion, as stated in the next theorem.

\begin{theorem}\label{th:corolario}
Let $\Sg=\C,$ $\D$  and  $u:\Sigma\to\R$ be a harmonic function.
Assume $u$ is non constant in case $\Sg=\D$.

Then, there exists a complete conformal minimal immersion
$Y=(Y_1,Y_2,Y_3):\Sg\to\R^3$ such that $Y_3=u$.
\end{theorem}

\begin{proof}
Fist of all notice that in case $\Sg=\C$ and $u$ is constant, the
plane $x_3=u$ satisfies the conclusion of the theorem.

Thus, let us assume that $u$ is non constant, and define the
holomorphic $1$-form $$\Phi_3=du+{\rm i} (\star du),$$ where
$\star$ denotes de Hodge operator. Since $\Sg$ is simply
connected we can write $\Phi_3=\phi_3(z) dz$ on $\Sg$. Then the
pair $(\phi_3,\Phi_3)$ are Weierstrass data on $\Sg$ and so they
define a conformal minimal immersion
$X=(X_1,X_2,X_3):\Sg\to\R^3$. Moreover, it is straightforward to
check that $X_3=u$. Now, applying Theorem \ref{th:main} to the
immersion $X$, any compact set $K\subset\Sg$ and any positive
$\ve$, we obtain an immersion fulfilling the statement of the
theorem.
\end{proof}

Other interesting (and immediate) consequence of Theorem
\ref{th:main} is a density type result. Let $u:\Sg\to\R$ be a non
constant harmonic function, and label $\cA_u=\{
X=(X_1,X_2,X_3):\Sg\to\R^3 \text{ conformal minimal
immersion}\;|\; X_3=u \}$. With this notation the following holds.

\begin{corollary}\label{co:dens}
Let $u:\Sg\to\R$ be a non constant harmonic function on
$\Sg=\C,\D$.

Then, complete immersions in $\cA_u$ are dense in $\cA_u$,
endowed with the topology of the uniform convergence on compact
sets.
\end{corollary}

Theorem \ref{th:main} can be applied to other geometric theories.
Maximal surfaces in the $3$-dimensional Lorentz-Minkowski space
$\L^3=(\R^3,dx_1^2+dx_2^2-dx_3^2)$ are spacelike surfaces with
vanishing mean curvature. There is a close connection between
minimal and maximal surfaces. Indeed, if
$X=\mbox{Re}\int(\Phi_1,\Phi_2,\Phi_3):\Sg\to\R^3$ is a conformal
minimal immersion defined on a simply connected surface $\Sg$,
then $$\widehat X=\mbox{Re}\int({\rm i}\Phi_1,{\rm
i}\Phi_2,\Phi_3):\Sg\to\L^3$$ is a conformal maximal immersion
(possibly with lightlike singularities), with the same third
coordinate function. See \cite{K,FLS} for more details on maximal
surfaces. Using this connection we can translate Theorem
\ref{th:corolario} to the Lorentzian setting.

\begin{corollary}
Let $\Sg=\C,$ $\D$  and  $u:\Sigma\to\R$ be a harmonic function.
Assume $u$ is non constant in case $\Sg=\D$.

Then, there exists a conformal maximal immersion (possibly with
lightlike singularities), $Y=(Y_1,Y_2,Y_3):\Sg\to\L^3$, such that
$Y_3=u$ and $Y$ is weakly complete in the sense of Umehara and
Yamada \cite{UY}.
\end{corollary}

Other geometrical objects related with minimal surfaces are null
curves (see for instance \cite{MUY}). By definition, a complex
curve $F=(F_1,F_2,F_3):\Sg\to\C^3$ is said to be a holomorphic
null curve if its coordinate functions are holomorphic and they
satisfy $$(F_1')^2+(F_2')^2+(F_3')^2=0,$$ where $'$ denotes the
complex derivative. Using Weierstrass representation, simply
connected minimal surfaces in $\R^3$ can be seen as the real part
of holomorphic null curves in $\C^3$, and conversely. Moreover,
the minimal surface and the associated holomorphic null curve
have the same metric. This allows us to prove the next result.

\begin{corollary}
Let $f:\Sg\to\C$ be a holomorphic function on $\Sg=\C,\D$. Assume
$f$ is non constant if $\Sg=\D$.

Then, there exists a complete holomorphic null curve
$F=(F_1,F_2,F_3):\Sg\to\C^3$ with $F_3=f$.
\end{corollary}

Finally, we would like to remark that our results are sharp in
the following sense. Recall that Nadirashvili's techniques give
complete conformal bounded minimal disks. Since our arguments are
inspired in his techniques, it could be expected that, starting
from a bounded harmonic function on the disk, one could obtain a
complete bounded minimal immersion having this function as a
coordinate function. However, the following proposition shows that
this is not possible in general.

\begin{proposition}\label{prop}
Let $X=(X_1,X_2,X_3):\D\to\R^3$ be a complete conformal minimal
immersion. Assume that $X_3$ can be extended smoothly to the
closed disk $\overline\D$.

Then,  $X_1$ and $X_2$ are unbounded on $\D$.
\end{proposition}

\begin{proof}
Let $\Phi=(\Phi_1,\Phi_2,\Phi_3)$ be the Weierstrass data of $X$
and write $\Phi_j=\phi_j(z)dz$, $j=1,2,3$.

Reasoning by contradiction, let us assume that $X_2$ is bounded.
Then, Bourgain's Theorem \cite{B} gives the existence of a real
number $0<\theta<2\pi$ such that $$\int_{0}^1 |\phi_2(r
e^{i\theta})| dr <\infty.$$

On the other hand, the assumption on $X_3$ guarantees that
$$\int_{0}^1 |\phi_3(r e^{i\theta})| dr <\infty.$$ Since $X$ is a
conformal map, it follows that $|\phi_1|^2\leq |\phi_2|^2 +
|\phi_3|^2$ and so, from the above inequalities we get
 $$\int_{0}^1 \|(\phi_1,\phi_2,\phi_3)(r e^{i\theta})\| dr
<\infty,$$ which contradicts the completeness of $X$.
\end{proof}


\def\refname{References}

\vspace*{1cm}

\noindent
{\bf Antonio Alarc\'{o}n}\\
Departamento de Geometr\'{\i}a y Topolog\'{\i}a, \\
Universidad de Granada, E-18071 Granada, Spain. \\
e-mail: alarcon@ugr.es.

\vspace*{0.5cm}

\noindent
{\bf Isabel Fern\'{a}ndez}\\ Departamento de Matem\'{a}tica Aplicada I,\\
Universidad de Sevilla, E-41012 Sevilla, Spain.\\ e-mail:
isafer@us.es

\end{document}